\documentclass[smallextended]{svjour3}

\usepackage{graphicx}
\usepackage{graphicx}
\usepackage{psfrag}
\usepackage{subfigure}
\usepackage{url}
\usepackage{color}
\usepackage{cite}
\usepackage{epsfig}
\usepackage{amsmath,amssymb,array}

\newcommand{\rset}{\mathbb{R}}

\newtheorem{assumption}[theorem]{Assumption}

\usecounter{algorithmenumi}

\providecommand{\norm}[1]{\lVert#1\rVert}

\begin{document}

\title{Adaptive inexact fast augmented Lagrangian  methods
for constrained convex optimization
\thanks{ First author's work has been funded  by the Sectoral Operational
Programme Human Resources Development 2007-2013 of the Ministry of
European Funds through the Financial Agreement
POSDRU/159/1.5/S/134398.}
}

\titlerunning{Adaptive inexact fast augmented Lagrangian  methods}
\author{Andrei Patrascu \and Ion Necoara  \and Quoc Tran-Dinh}

\authorrunning{Patrascu, Necoara, Tran-Dinh}

\institute{A. Patrascu and I. Necoara \at
Automatic Control and Systems Engineering Department, University
Politehnica Bucharest, 060042 Bucharest, Romania,
\email{\{andrei.patrascu,ion.necoara\}@acse.pub.ro} \\
Q. Tran-Dinh \at
Laboratory for Information and Inference Systems, EPFL,  Lausanne, Switzerland, \email{quoc.trandinh@epfl.ch}.
}

\date{Received: May 2015 / Accepted: date}

\maketitle

\begin{abstract}
In this paper we analyze several inexact fast augmented Lagrangian
methods for solving linearly constrained convex optimization
problems. Mainly, our methods rely on the combination of
excessive-gap-like smoothing technique developed in
\cite{Nes:excessive:05} and the newly introduced
inexact oracle framework from \cite{DevGli:14}.
We analyze several algorithmic instances with constant and adaptive smoothing
parameters and derive total computational  complexity results in terms 
of projections onto a simple primal set. For the basic inexact fast 
augmented Lagrangian algorithm we obtain the overall computational 
complexity of order  $\mathcal{O}\left(\frac{1}{\epsilon^{5/4}}\right)$, 
while for the adaptive variant we get  $\mathcal{O}\left(\frac{1}{\epsilon}\right)$, 
projections onto a primal set in order to obtain an $\epsilon-$optimal solution 
for our original problem.

\end{abstract}

\section{Introduction}
Large-scale constrained convex optimization models are convenient tools to formulate several practical problems in modern engineering, statistics, and economical applications.
Several important applications in such fields can be modeled into this class of problems such as:
linear (distributed) model predictive control \cite{NecNed:13,
NedNec:14}, network utility maximization \cite{BecNed:14}, or
compressed sensing \cite{AybIye:12,QuoSav:12}.
However, solving large-scale optimization problems is still a challenge in many applications due
to the limitations of computational devices and computer systems, as well as the limitations of existing algorithm techniques.

In the recent optimization literature, the primal-dual first order methods have gained a great attention due to their low complexity-per-iteration and  
their flexibility of handling linear operators, constraints and nonsmoothness.
In the constrained case, when complicated constraints are present, many first order algorithms are combined
with duality or penalty strategies in order to process them.
Amongst the most popular primal-dual approaches, which behaves very well in practice, the augmented 
Lagrangian (AL) strategy was extensively studied, e.g., in \cite{Roc:76, LanMon:14, AybIye:12, NecPat:15, NedNec:14, QuoCev:14, HeTao:12, HeYan:04}.
It is well known that the AL smoothing technique is a multi-stage strategy implying successive computations of solutions for certain primal-dual subproblems.
In general, these subproblems do not have closed form solutions, which makes the AL strategies inherently related to inexact first-order algorithms.
In these settings, most of the complexity results regarding AL and fast AL algorithms are given under inexact first order information \cite{LanMon:14,
AybIye:13,NecPat:15,NedNec:14,QuoCev:14}.
To our best knowledge, the complexity estimates for the fast AL methods have been studied only in \cite{NedNec:14,QuoCev:14}.
However, in these papers only outer complexity estimates are provided.
It is clear that the outer complexity estimates in AL methods do not  take into account the complexity-per-iteration.
In many situations we can choose the smoothing parameters in order to perform only one outer iteration, see, e.g., 
\cite{AybIye:13,NecPat:15,LanMon:14}, but the overall complexity estimate is still high due to the high complexity-per-iteration.
Trading off these quantities is a crucial question in the implementation of AL methods.

In this paper we aim at improving  the overall iteration  complexity of fast AL methods using the inexact oracle framework
developed in \cite{DevGli:14}  based on  a simple inner accuracy update and excessive gap-like fast gradient algorithms \cite{Nes:excessive:05}.
By using the inexact oracle framework, we are able to provide the overall computational complexity of an inexact fast AL method with constant
smoothing parameter and of its adaptive variant.

\vspace{1ex}
\noindent \textbf{Our contributions.}
In this paper we analyze the computational complexity of Inexact Fast Augmented Lagrangian (IFAL)
method with constant and adaptive smoothing parameters.
Using the inexact oracle  framework \cite{DevGli:14}, our approach allow us to obtain clean and intuitive complexity results and, moreover, it
facilitates the derivation of the overall computational complexity of our methods.
\vspace{-0.75ex}
\begin{itemize}
\item[$(a)$] For the basic  IFAL method with constant smoothing
parameter, we derive $\mathcal{O}\left(\frac{1}{\sqrt{\epsilon}}\right)$ outer complexity estimates
corresponding to simple  inner accuracy updates. We also derive the
overall computational complexity of order $\mathcal{O}\left(\frac{1}{\epsilon^{5/4}} \right)$ projections onto a primal
set, in order to obtain an $\epsilon-$optimal solution in the sense of the objective residual and feasibility violation.

\item[$(b)$] Then, we show that for an optimal choice for the smoothing parameters we need to perform only \textit{one} outer iteration.
Based on this result, we introduce an adaptive IFAL method with variable smoothing parameters for which we prove an overall computational complexity of order $\mathcal{O}\left(\frac{1}{\epsilon} \right)$ projections.

\item[$(c)$] We show that our adaptive variant of inexact fast AL method is implementable, i.e. it is based on computable stopping
criteria. Moreover, we compare our results with other complexity estimates for AL methods from literature and highlight the
advantageous features of our methods.
\end{itemize}

\noindent\textbf{Paper organization.}
The rest of this paper is organized as follows.
In Section 2 we define our optimization
model and introduce some preliminary concepts related to  duality
and inexact oracles. In Section 3 we introduce the inexact fast
AL method with constant smoothing parameter and
analyze its computational complexity. To improve this complexity, in
Section 4 we study an adaptive parameter variant and provide its
complexity estimate. Finally, in Section 5, we compare our results
with other complexity results on AL methods from the literature.

\noindent\textbf{Notations.} We work in the Euclidean space $\rset^n$
composed by column vectors. For $u,v \in \rset^n$ we denote the
 inner product $\langle u,v \rangle = u^T v$ and the Euclidean norm $\left \| u \right \|=\sqrt{\langle u, u
\rangle}$.
For any matrix $G$ we denote by $\norm{G}$ its spectral norm.

%%%%%%%%%%%%%%%%%%%%%%%%%%%%%%%%%%%%%%%%%%%%%%%%%%%%%%%%%%%%%%
%%%%%%%%%%%%%%%%%%%%%%%%%%%%%%%%%%%%%%%%%%%%%%%%%%%%%%%%%%%%%%

\section{Problem formulation and preliminaries}
We consider the following linearly constrained convex optimization problem:
\begin{align}\label{problem}
 f^* = \left\{\begin{array}{ll}\min\limits_{u \in U} & f(u) \\
\text{s.t.} & Gu+g = 0,
 \end{array}\right.
\end{align}
where $f:\rset^n \to \rset\cup\{+\infty\}$ is a proper, closed and convex function, $U \subseteq \rset^n$ is a nonempty, closed and convex set, $G \in \rset^{m \times n}$, and $g\in\rset^m$.
We use $U^*$ to denote the optimal solution set of set of \eqref{problem}, which will be assumed to be nonempty.

The goal of this paper is to analyze the optimization model \eqref{problem}  and to develop new inexact Augmented Lagrangian methods with convergence guarantees for approximately solving \eqref{problem}.
For this purpose, we require the following blanket assumptions which are assumed to be valid
throughout the paper without recalling them in the sequel:

\begin{assumption}\label{all_assumptions}
\begin{enumerate}
\item[$(i)$] The solution set $U^{*}$ is nonempty. The feasible set $U$ is  bounded and simple $($i.e.,
the projection onto $U$ can be computed in a closed form or in polynomial time$)$.
\item[$(ii)$] There exist a bounded optimal Lagrange multiplier $x^* \in \rset^m$.
\item[$(iii)$] The objective function $f$ has the Lipschitz gradient  with the Lipschitz constant $L_f > 0$, i.e.:
\begin{equation*}
 \norm{\nabla f(u) -  \nabla f(v)} \le L_f \norm{u-v}, \qquad u,v \in \rset^n.
\end{equation*}
\end{enumerate}
\end{assumption}

If $f$ is strongly convex, it is well-known that the Lagrangian dual function associated with the linear constraint $Gu + g = 0$ has the Lipschitz gradient.
This setting has been extensively studied in the literature (see, e.g., \cite{BecNed:14,NecPat:14,NecPat:15,QuoCev:14,NecNed:13}).
Thus, in the rest of our paper we only assume $f$ to be a convex function and not necessarily strongly convex.
Due to the constraint, it is clear that the dual function $d$ defined by $d(x) = \min\limits_{u\in U}\{ f(u) + \langle x, Gu + g\rangle \}$ is, in general,
nonsmooth and concave, which induces the difficulties in the application of usual first order methods to the dual.
Therefore, various (dual) subgradient schemes have been developed, with iteration complexities of order $\mathcal{O}\left(\frac{1}{\epsilon^2}\right)$
\cite{NedOzd:09,Nes:14}.

Under additional mild assumptions, we aim in this paper at improving the iteration complexity required for solving the linearly constrained optimization problem \eqref{problem}.
Our approach relies  on the combination between smoothing techniques and duality \cite{NecSuy:08,Nes:excessive:05,QuoNec:15}.

First, we briefly recall the Lagrangian duality framework as follows.
The Lagrangian function and the dual function associated
to the convex problem \eqref{problem} are defined by:
\begin{equation*}
  \mathcal{L}(u,x) = f(u) + \langle x, Gu+g\rangle \quad \text{and} \quad  d(x) = \min\limits_{u \in U} \; \mathcal{L}(u,x).
\end{equation*}
From Assumption \ref{all_assumptions}$(ii)$ it follows that the
convex problem \eqref{problem} is equivalent with solving the dual
formulation, i.e.:
\begin{align}
\label{classic dual} f^* = \max_{x \in \rset^m} d(x) \qquad \left(=
\max_{x \in \rset^m} \min\limits_{u \in U} \mathcal{L}(u,x) \right).
\end{align}
Our goal in this paper is to find an
approximate primal solution for the optimization problem
\eqref{problem}. Therefore, we introduce the following definition:

\begin{definition}\label{opt_point}
Given a desired accuracy $\epsilon>0$, the  point ${u}_\epsilon\in U$ is called an \textit{$\epsilon$-optimal solution} for the primal
problem \eqref{problem} if it satisfies:
\begin{equation*}
f({u}_\epsilon ) - f^* \le \epsilon \quad \text{and} \quad \norm{G{u}_\epsilon +g} \le \epsilon.
\end{equation*}
\end{definition}
This set of optimality criteria has been also adopted by Rockafellar
in \cite{Roc:76} in the context of classical augmented Lagrangian methods.
Moreover, the above criteria has been also used by Nesterov in
\cite{Nes:15}  for the analysis  of primal-dual subgradient methods.
It can be easily observed that once we have an $\epsilon$-feasible point, i.e., $\Vert Gu_{\epsilon} + g\Vert \leq \epsilon$, then we can also obtain a lower bound on $f({u}_\epsilon ) - f^*$. Indeed, we have the relation:
\begin{align*}
f^* = \min_{u \in U} f(u) + \langle x^*, Gu+g\rangle \leq f({u}_\epsilon) + \|x^*\| \norm{G{u}_\epsilon +g},
\end{align*}
and thus $f({u}_\epsilon ) - f^* \geq - \|x^*\| \norm{G{u}_\epsilon +g}$.
Moreover, from a practical point of view, it is sufficient to find an $\epsilon$-feasible point ${u}_\epsilon$ satisfying
$f({u}_\epsilon ) - f^* \le \epsilon$.

Given the max-min formulation \eqref{classic dual},  one can intuitively consider a double smoothing of the convex-concave Lagrangian function. Regarding the dual function, one of the most widely known smoothing strategies for obtaining an approximate smooth dual function with Lipschitz continuous gradient is the augmented Lagrangian (AL) smoothing \cite{AybIye:13,Roc:76,NecPat:15,NedNec:14,QuoCev:14}.
Thus, we combine the AL technique with a smooth approximation of the primal function and define:
\begin{align}\label{inner_subproblem}
\mathcal{L}_{\mu\rho} (u,x) &= f(u) +\langle x, Gu+g\rangle +  \frac{\rho}{2}\norm{Gu+g}^2 - \frac{\mu}{2}\norm{x}^2,
\end{align}
where $\mu, \rho > 0$ are two  smoothness parameters.
Clearly, we have the relation $\lim\limits_{\mu, \rho \to 0} \mathcal{L}_{\mu\rho} (u,x) = \mathcal{L}_{} (u,x)$ and, in addition, $\mathcal{L}_{\mu\rho}(\cdot,x)$ has Lipschitz continuous gradient with the Lipschitz constant $L_{\mathcal{L}}  = L_f + \rho \norm{G}^2$ for any fixed $x$.
Based on this approximation of the true Lagrangian function we also define two smooth  approximations of the primal and dual functions $f$ and $d$, respectively:
\begin{align*}
d_{\rho} (x) &= \min_{u \in U} \; \mathcal{L}_{0\rho}(u,x) \qquad
\left( = \min_{u \in U} f(u) +\langle x, Gu+g\rangle + \frac{\rho}{2}\norm{Gu+g}^2 \right), \\
f_{\mu}(u)  & = \max_{x \in \rset^m} \mathcal{L}_{\mu 0}(u,x) \qquad \left( = f(u) + \frac{\mu}{2}\norm{Gu+g}^2 \right).
\end{align*}
Let  define the optimal solutions of the two previous optimization problems:
\begin{align*}
u_{\rho}(x) \!=\! \arg\!\min_{u \in U} \mathcal{L}_{0\rho}(u,x) \quad \text{and} \quad
x_{\mu}(u)  \!=\! \arg\!\max_{x \in \rset^m} \mathcal{L}_{\mu 0}(u,x)  \left(\!= \frac{1}{\mu}(Gu + g) \!\right).
\end{align*}
Clearly, both functions $f_{\mu}$ and $d_{\rho}$ are smooth
approximations of $f$ and $d$, respectively. In particular, we
observe that the smoothed primal function $f_{\mu}$ has Lipschitz continuous gradient with the Lipschitz
constant $L_{f_\mu}  = L_f + \mu \norm{G}^2$. Moreover,  the
smoothed dual function $d_{\rho}$ is concave and its gradient $\nabla
d_{\rho}(x) = G u_{\rho}(x)  + g$ is Lipschitz continuous with the Lipschitz
constant $L_{d_\rho} = \frac{1}{\rho}$. We emphasize again that, in most practical cases, $u_{\rho}(x)$ cannot be computed exactly, but within a pre-specified accuracy, which leads us to consider the inexact oracle  framework introduced in \cite{DevGli:14} for the analysis of inexact first order algorithms.

Recall that a smooth function $\phi:Q \to \rset$ is equipped with a \textit{first-order $(\delta,L)$-oracle} if for any $y \in Q$ we can compute $(\phi_{\delta,L}(y), \nabla \phi_{\delta,L}(y)) \in \rset^{ }\times \rset^n$ such that the following bounds hold on $\phi$ (so-called inexact descent lemma)
\cite{DevGli:14}:
\begin{equation}
\label{deltaL} 0 \le \phi(x) - \left(\phi_{\delta,L}(y) + \langle
\nabla  \phi_{\delta,L}(y), x-y\rangle \right) \le
\frac{L}{2}\norm{x-y}^2 + \delta \quad \forall x,y \in Q.
\end{equation}

If we define $\tilde{u}_{\rho}(x) \in U$ as the inexact solution of the inner subproblem in $u$ satisfying:
\begin{equation}\label{func_approx}
 0 \le  \mathcal{L}_{0\rho}(\tilde{u}_{\rho}(x),x) - d_{\rho}(x) \le \delta,
\end{equation}
then using the notation $\nabla_x
\mathcal{L}_{0\rho}(\tilde{u}_{\rho}(x),x) = G
\tilde{u}_{\rho}(x)+g$,  we are able to provide the following
important auxiliary result,  whose  proof can be found  in
\cite{NecPat:15}:

\begin{lemma}[\!\!\cite{NecPat:15}]\label{th_inexact_oracle}
Let $\delta>0$ and $\tilde{u}_{\mu}(x) \in U$ satisfy \eqref{func_approx}. Then, for all $x, y$, we have:
\begin{align}\label{inexact_oracle}
0 \!\le\! \mathcal{L}_{0\rho}&(\tilde{u}_{\rho}(y),y) +   \langle \nabla_x \mathcal{L}_{0\rho}(\tilde{u}_{\rho}(y),y), x \!-\! y \rangle -
d_{\rho}(x) \le L_{\text{d}_\rho}\norm{x \!-\! y}^2 + 2\delta.
\end{align}
\end{lemma}

The relation \eqref{inexact_oracle} implies that  $d_{\rho}$ is equipped with a $(2\delta, 2L_{\text{d}_\rho})$-oracle with  $\phi_{\delta,L} (x) \!=\!
\mathcal{L}_{0\rho}(\tilde{u}_{\rho}(x),x)$ and $\nabla
\phi_{\delta,L} (x) = \nabla_x
\mathcal{L}_{0\rho}(\tilde{u}_{\rho}(x),x) $ $= G
\tilde{u}_{\rho}(x)+g$. It is important to note that the analysis
considered in \cite{NedNec:14,QuoCev:14,QuoNec:15} requires to solve
the inner problem with higher accuracy of order
$\mathcal{O}(\delta^2)$, i.e.:
$$ \mathcal{L}_{0\rho}(\tilde{u}_{\rho}(x),x) - d_{\rho}(x) \le
\mathcal{O}(\delta^2)$$ in order to ensure bounds on $d_{\rho}(x)$
of the form:
\begin{align*}
0 \le \mathcal{L}_{0\rho}(\tilde{u}_{\rho}(y),y) + \langle \nabla_x
\mathcal{L}_{0\rho}(\tilde{u}_{\rho}(y),y),& x - y
\rangle - d_{\rho}(x) \\
&\le \frac{L_{\text{d}_\rho}}{2}\norm{x - y}^2 + \left(1 +
\sqrt{2L_{\mathcal{L}}}D_U\right)\delta,
\end{align*}
where $D_U$ is the diameter of the bounded convex domain $U$.
It is obvious that our approach in this paper is less conservative, requiring to solve the inner problem with less accuracy than in \cite{NedNec:14,QuoCev:14,QuoNec:15}.
As we will see in the sequel, this will also have an impact on the total complexity of our method compared to those in the previous papers.

%%%%%%%%%%%%%%%%%%%%%%%%%%%%%%%%%%%%%%%%%%%%%%%%%%%%%%%%%%%%%%%%%
%%%%%%%%%%%%%%%%%%%%%%%%%%%%%%%%%%%%%%%%%%%%%%%%%%%%%%%%%%%%%%%

\section{Inexact fast augmented Lagrangian method}
In this section we propose an augmented Lagrangian smoothing strategy similar to the excessive gap technique introduced in
\cite{Nes:excessive:05,QuoNec:15,QuoCev:14}.
Typical excessive gap strategies are based on primal-dual fast gradient methods, which maintain at each iteration some excessive gap inequality.
Using this inequality, the convergence of the outer loop of the algorithm is naturally determined.

In this paper, we use an excessive gap-like inequality, which holds at each outer iteration of our algorithm.
Given the dual smoothing parameter $\rho$, inner accuracy $\delta$ and outer accuracy $\epsilon$, we further develop an Inexact Fast
Augmented Lagrangian (IFAL) algorithm for solving \eqref{problem}:
%Note that the IFAL algorithm is similar to the one introduced and analyzed 
%in \cite{QuoNec:15,QuoCev:14}, where the update rules for  $\tau_k$ are derived below.
\vspace{0.5ex}
\begin{center}
\framebox{
\parbox{11.5cm}{
%\begin{center}
\noindent \textbf{IFAL($\rho, \epsilon $) Algorithm}
\begin{itemize}
\item[]\textbf{Initialization:} Give $ u^0 \in U, x^0 \in \rset^m$ and $\mu_0 > 0$.
\item[]\textbf{Iterations:} For $k = 0, 1, \dots$, perform the following steps:
\begin{enumerate}
\item $\hat{x}^{k} = (1- \tau_k) x^k + \frac{\tau_k}{\mu_k} (Gu^k + g)$
\item Find $\tilde{u}_{\rho}(\hat{x}^k)$ such that $\mathcal{L}_{0\rho}(\tilde{u}_{
\rho}(\hat{x}^k), \hat{x}^k) - d_{\rho}(\hat{x}^k) \le \delta_k$
\item $u^{k+1}= (1- \tau_k) u^k + \tau_k \tilde{u}_{\rho}(\hat{x}^k)$
\item $x^{k+1} = \hat{x}^k + \rho(G \tilde{u}_{\rho}(\hat{x}^k) + g)$
\item Set $\mu_{k+1} = (1-\tau_k)\mu_k$
\item If a stopping criterion holds then STOP and \textbf{return}:
 $(u^k, x^k)$,
\end{enumerate}
\item[]\textbf{End}
\end{itemize}
}}
\end{center}

The update rules for  $\tau_k$ are derived below. We define as in \cite{Nes:excessive:05,QuoNec:15,QuoCev:14} the smoothed duality gap $\Delta_k = f_{\mu_k}(u^k) - d_{\rho}(x^k)$.
Based on this smoothed duality gap $\Delta_k$, we further provide a descent inequality, which will facilitate  the derivation of a simple inner accuracy update and of the total complexity of the IFAL algorithm.
%which is different from the ones derived in
%\cite{QuoCev:14,QuoNec:15}. Although we use certain facts from the proof of \cite[Theorem 5.1]{QuoCev:14}, our approach, based on Lemma
%\ref{th_inexact_oracle},  leads to a clean and intuitive result which strongly facilitates the derivation of a simple inner accuracy
%update and of the total complexity of the IFAL algorithm.

%% Theorem 2.
\begin{lemma}
Let $\rho>0$ and $\{ (u^k,x^k) \}$ be the sequences generated by the IFAL Algorithm.
If  the parameter $\tau_k \in (0, 1)$ satisfies $\frac{\tau_k^2}{1 - \tau_k} \le \mu_{k} \rho$ for all $k \geq 0$, then the following excessive gap inequality
holds:
\begin{equation}\label{lemma_delta}
  \Delta_{k+1} \le (1- \tau_k)\Delta_k + 2\delta_k.
\end{equation}
\end{lemma}

\begin{proof}
Firstly, we observe that, from strong convexity of $\norm{\cdot}^2$ results:
\begin{align}\label{estimate_deltak}
\Delta_k & = \max_{x \in \rset^m} f(u^k) + \langle x, Gu^k +  g \rangle - \frac{\mu_k}{2}\norm{x}^2 - d_{\rho}(x^k)  \\
& \ge f(u^k) \!+\! \langle x, Gu^k + g \rangle \!-\! \frac{\mu_k}{2}\norm{x}^2 \!+\! \frac{\mu_k}{2}\norm{x - x_{\mu_k}(u^k)}^2 \!-\! d_{\rho}(x^k) \;\;\;\; \forall x \in \rset^m.
\nonumber
\end{align}
Secondly, given any $x \in \rset^m$, from the  definition of $\mathcal{L}_{0\rho}$, note that we have:
\begin{align}\label{expression_linear_dual}
\mathcal{L}_{0\rho}(\tilde{u}_{\rho}(\hat{x}^k),\hat{x}^k) + \langle \nabla_x \mathcal{L}_{0\rho}(\tilde{u}_{\rho}(\hat{x}^k), \hat{x}^k), x - \hat{x}^k \rangle & = f(\tilde{u}_{\rho}(\hat{x}^k)) + \langle x, G \tilde{u}_{\rho}(\hat{x}^k) + g\rangle \nonumber\\
& + \frac{\rho}{2}\norm{G \tilde{u}_{\rho}(\hat{x}^k) + g}^2.
\end{align}
Multiplying both sides of \eqref{estimate_deltak} with $1 - \tau_k$
and combining with \eqref{expression_linear_dual}, we obtain:
\begin{align*}
(1- \tau_k) \Delta_k  &\geq (1-\tau_k)\left(f(u^k) + \langle x, Gu^k + g \rangle - \frac{\mu_k}{2}\norm{x}^2 + \frac{\mu_k}{2}\norm{x - x_{\mu_k}(u^k)}^2\right. \\
& \left. - d_{\rho}(x^k) \right)   + \tau_k \Big ( f(\tilde{u}_{\rho}(\hat{x}^k)) + \langle x, G \tilde{u}_{\rho}(\hat{x}^k) + g\rangle + \frac{\rho}{2}\norm{G \tilde{u}_{\rho}(\hat{x}^k) + g}^2 \\
& - \mathcal{L}_{0\rho}(\tilde{u}_{\rho}(\hat{x}^k),\hat{x}^k) - \langle \nabla_x \mathcal{L}_{0\rho}(\tilde{u}_{\rho}(\hat{x}^k), \hat{x}^k), x - \hat{x}^k \rangle \Big ).
\end{align*}
Using the convexity of $f$ we further have:
\begin{align*}
(1- \tau_k) \Delta_k &\ge f((1- \tau_k)u^k + \tau_k \tilde{u}_{\rho}(\hat{x}^k))  + \langle x, G\left((1- \tau_k)u^k + \tau_k\tilde{u}_{\rho}(\hat{x}^k)\right)  + g \rangle \\
&- \frac{\mu^{k+1}}{2}\norm{x}^2 - (1- \tau_k) d_{\rho}(x^k)   + \frac{\mu^{k+1}}{2}\norm{x - x_{\mu_k}(u^k)}^2\\
& - \tau_k\left(\mathcal{L}_{0\rho}(\tilde{u}_{\rho}(\hat{x}^k),\hat{x}^k) +
\langle \nabla_x \mathcal{L}_{0\rho}(\tilde{u}_{\rho}(\hat{x}^k), x-
\hat{x}^k)\rangle \right).
\end{align*}
Further, from the left  hand side of inexact descent lemma
\eqref{inexact_oracle} we get:
\begin{align*}
(1\!-\! \tau_k) \Delta_k &\ge f(u^{k+1}) + \langle x, Gu^{k+1} + g\rangle - \frac{\mu^{k+1}}{2}\norm{x}^2 - \mathcal{L}_{0\rho}(\tilde{u}_{\rho}(\hat{x}^k), \hat{x}^k) \\
&{\!\!}  - \langle \nabla_x \mathcal{L}_{0\rho}(\tilde{u}_{\rho}(\hat{x}^k), \hat{x}^k), (1-\tau_k)x^k + \tau_k x  - \hat{x}^k\rangle + \frac{\mu^{k+1}}{2}\norm{x - x_{\mu_k}(u^k)}^2\\
& \ge f(u^{k+1}) + \langle x, G u^{k+1} + g \rangle - \frac{\mu^{k+1}}{2}\norm{x}^2 - \mathcal{L}_{0\rho}(\tilde{u}_{\rho}(\hat{x}^k), \hat{x}^k) \\
& - \max\limits_{z:= (1-\tau_k)x^k + \tau_k x}  \langle
\nabla_x\mathcal{L}_{0\rho}(\tilde{u}_{\rho}(\hat{x}^k),\hat{x}^k),
z - \hat{x}^k  \rangle - \frac{\mu^{k+1}}{2\tau^2_k}\norm{z -
\hat{x}}^2.
\end{align*}
Using the assumption that  $\frac{\tau_k^2}{1 - \tau_k} \le \mu_k
\rho$ and the right hand side of inexact descent lemma
\eqref{inexact_oracle} we further derive:
\begin{equation*}
 (1- \tau_k) \Delta_k \ge f(u^{k+1}) + \langle x, G u^{k+1} + g \rangle - \frac{\mu^{k+1}}{2}\norm{x}^2 - d_{\rho}(x^{k+1}) - 2\delta_k \quad \forall x \in \rset^m.
\end{equation*}
Choosing $x = \frac{1}{\mu^{k+1}}(Gu^{k+1}+g)$ we obtain our result.
\qed
\end{proof}

Similar excessive gap inequalities have  been proved in \cite{Nes:excessive:05,QuoNec:15,QuoCev:14} and then used for analyzing the convergence rate of excessive gap algorithms. However, the main results in \cite{Nes:excessive:05,QuoNec:15,QuoCev:14} only concern with the outer iteration complexity and do not take into account the necessary inner computational effort for finding $\tilde{u}_{\rho}(\cdot)$,
since it is very difficult to estimate this quantity using the approaches presented in these papers.

In the sequel, based on our approach, we provide the total computational complexity of the IFAL algorithm (including inner complexity) for
attaining an $\epsilon-$optimal solution of problem \eqref{problem}.
First, we notice  that  if we assume  $ \mu_0 = \frac{4}{\rho }$, by taking into account that $\mu_k = \mu_0 \prod_{j=0}^k (1-\tau_j)$, a
simple choice of sequence $\tau_k$ satisfying $\frac{\tau_k^2}{1 - \tau_k} \le \mu_k \rho$ is given by $\tau_k = \frac{2}{k+3}$.

% Theorem 2.
\begin{theorem}\label{th_outer_rate_conv}
Let $\rho,\epsilon>0, \mu_0 = \frac{4}{\rho}$ and $\tau_k = \frac{2}{k+3}$.
Let $\{ (u^k,x^k) \}$ be  the sequences generated by the IFAL$(\rho,\epsilon)$ Algorithm.
If we choose the inner accuracy $\delta_k = \frac{\epsilon}{2(k+3)}$, then the following estimates on
objective residual and the feasibility violation hold:
\begin{equation}\label{eq:opt_ests}
\left\{\begin{array}{ll}
f(u^k) - f^* &\leq \frac{\Delta_0}{(k\!+\!1)(k\!+\!2)} \!+\! \frac{\epsilon}{2},\vspace{1ex}\\
\norm{G u^k + g} &\le \frac{4 \norm{x^*} \!+\! \sqrt{8 \rho \Delta_0}}{\rho(k+1)(k+2)} +
\frac{(8\epsilon)^{1/2}} {\rho^{1/2}(k+1)}.
\end{array}\right.
\end{equation}
\end{theorem}

\begin{proof}
From \eqref{lemma_delta} it can be derived that:
\begin{align}\label{rate_conv_aux}
 \Delta_{k+1} &\le \Delta_0  \prod_{j=0}^k (1-\tau_j) + 2\delta_k + 2\sum\limits_{i=1}^k \delta_{k-i} \prod\limits_{j=k-i+1}^k (1-\tau_j) \nonumber \\
 & = \Delta_0 \prod\limits_{j=0}^k (1-\tau_j) + 2\prod_{j=0}^k (1- \tau_j) \Bigg(\sum\limits_{i=0}^k \frac{\delta_i}{\prod_{s = 0}^i (1-\tau_s) } \Bigg).
\end{align}
Observing that $\prod_{j=0}^k (1- \tau_j) = \frac{2}{(k+2)(k+3)}$,
we can further bound the cumulative error as follows:
\begin{align*}
 2\prod_{j=0}^k (1- \tau_j) \Bigg(\sum\limits_{i=0}^k \frac{\delta_i}{\prod\limits_{s = 0}^i (1-\tau_s) }\Bigg)
&= \frac{4}{(k+2)(k+3)} \sum\limits_{i=0}^k \frac{(i+2)(i+3) \delta_i}{2} \\
& = \frac{\epsilon}{2} \frac{k(k+5)}{(k+2)(k+3)} \le
\frac{\epsilon}{2}.
\end{align*}
Using this bound and  $f^* \ge d_{\rho}(x)$ for any $x \in \rset^m$,
from \eqref{rate_conv_aux} we have:
\begin{equation}\label{rate_conv}
f(u^k) - f^* \le f_{\mu_k}(u^k) - d_{\rho}(x^k) = \Delta_{k} \le
\frac{\Delta_0}{(k+1)(k+2)} + \frac{\epsilon}{2}.
\end{equation}
On the other hand, using the KKT conditions of  \eqref{problem}, we have:
\begin{align}\label{feas_relation_aux}
f_{\mu_k}(u^k) - f^* &  = f(u^k) + \frac{\mu_k}{2}\norm{Gu^k+g}^2 - f^* \nonumber \\
& \ge \frac{\mu_k}{2}\norm{G u^k +g}^2 + \langle \nabla f(u^*), u^k -u^*\rangle \nonumber \\
& \ge \frac{1}{2\mu_k}\norm{Gu^k + g}^2 - \norm{x^*}\norm{G u^k +g},
\end{align}
which is the first estimate in \eqref{eq:opt_ests}
Combining \eqref{rate_conv} and \eqref{feas_relation_aux}, we
obtain:
\begin{align*}
 \norm{Gu^k + g} &\le \mu_k \norm{x^*} + \left[\frac{2 \mu_k \Delta_0}{(k+1)(k+2)} + 2\mu_k\epsilon\right]^{1/2} \\
 & \le \frac{\mu_0 \norm{x^*} + (2\mu_0 \Delta_0)^{1/2}}{(k+1)(k+2)} + \frac{(2\mu_0 \epsilon)^{1/2}}{k+1},
\end{align*}
which is the second estimate of \eqref{eq:opt_ests}.
\qed
\end{proof}

Now, we provide the overall computational complexity of the IFAL Algorithm in terms of the number of projections on the primal set $U$.
% Theorem 3.
\begin{theorem} \label{th_total_compl_nonoptimal}
Let $\rho, \epsilon> 0, \mu_0 = \frac{4}{\rho}$ and $\tau_k  = \frac{2}{k+ 3}$.
Assume that at each outer iteration $k$ of the IFAL($\rho,\epsilon$) Algorithm, Nesterov's optimal method \cite{Nes:04} is called  for computing an
approximate solution $\tilde{u}_{\rho}(\hat{x}^k)$ of the inner subproblem \eqref{inner_subproblem} such that
$\mathcal{L}_{0\rho}(\tilde{u}_{\rho}(\hat{x}^k),\hat{x}^k) - d_{\rho}(\hat{x}^k) \le \delta_k \quad (= \frac{\epsilon}{2(k+3)})$.
Then, for attaining an $\epsilon-$optimal solution in the sense of Definition \ref{opt_point}, the
IFAL Algorithm performs at most:
\begin{equation*}
\left \lfloor \frac{16 \gamma D_U \left[2(L_f + \rho \norm{G}^2) \right]^{1/2}}{\epsilon^{5/4}} \right \rfloor
\end{equation*}
projections on the simple set $U$, where the constant $\gamma$ has
the following expression: $\gamma = \max \left\{(\epsilon/3)^{1/2},
(32/\rho)^{1/2}, \left(\frac{8\norm{x^*}}{\rho} + 2\sqrt{\frac{4
\epsilon }{3\rho}} \right)^{1/2} \right\}$.
\end{theorem}

\begin{proof}
Note that from Theorem \ref{th_outer_rate_conv} we observe that for
attaining an $\epsilon-$optimal solution, the number of outer
iterations $N^{\text{out}}$ must satisfy:
\begin{equation*}
 N^{\text{out}} \le \frac{1}{\epsilon^{1/2}} \max \left\{ \left( 2\Delta_0\right)^{1/2}, \left( 8 \mu_0 \right)^{1/2}, \left(2\mu_0 \norm{x^*} + 2\sqrt{2\mu_0\Delta_0} \right)^{1/2} \right\} = \frac{\gamma}{\epsilon^{1/2}}.
\end{equation*}
On the other hand, at each outer iteration $k$,  Assumption
\ref{all_assumptions}$(iii)$ implies that Nesterov's optimal method
\cite{Nes:04} applied on the inner subproblem
\eqref{inner_subproblem} performs:
\begin{equation}\label{inner_Nesterov}
   N^{\text{in}}_k =  2 D_U \sqrt{ \frac{L_f + \rho \norm{G}^2}{\delta_k} }  = 2D_U \sqrt{\frac{2(L_f + \rho \norm{G}^2)(k+3)}{\epsilon}}
\end{equation}
inner iterations (i.e., projections on $U$).
Using this estimate we can easily derive the total number of projections on the simple set $U$ necessary for attaining an
$\epsilon-$optimal point:
\begin{align*}
  \sum\limits_{k=0}^{N^{\text{out}}} N^{\text{in}}_k &= 2D_U \left[\frac{2(L_f + \rho \norm{G}^2)}{\epsilon}\right]^{1/2} \sum\limits_{k=0}^{N^{\text{out}}} (k+3)^{1/2} \\
 & \le 2 D_U \left[ \frac{2(L_f + \rho \norm{G}^2)}{\epsilon} \right]^{1/2} (N^{\text{out}} + 3)^{3/2} \\
 & \le \frac{16 \gamma D_U \left[2(L_f + \rho \norm{G}^2) \right]^{1/2}}{\epsilon^{5/4}}.
\end{align*}
The theorem is proved.
\qed
\end{proof}

\begin{remark}\label{remark1}
We observe that $\Delta_0 \le \tilde{\Delta}_0$, where $\tilde{\Delta}_0 = f_{\mu_0}(u^0) -
\mathcal{L}_{0\rho}(\tilde{u}_{\rho}(x^0),x^0) + \delta_0$ can be
computed explicitly. Then,  using this upper bound in our estimates,
 makes  Algorithm IFAL implementable, i.e., the algorithm stops with an  $\epsilon-$optimal solution  at the outer iteration
$N^{\text{out}}$ provided that the following two computable conditions hold:
\begin{equation*}
\norm{G u^k + g} \le \epsilon \quad  \text{and} \quad N^{\text{out}} \ge \left(\frac{2 \tilde{\Delta}_0}{\epsilon}\right)^{1/2}.
\end{equation*}
Moreover, as suggested, e.g., in \cite{QuoNec:15,QuoCev:14}, if we choose the primal-dual starting points $u^0 = \tilde{u}_{\rho}(0)$ and $ x^0 =
\frac{1}{\mu_0}(Gu^0+g)$, respectively,  then we even have $\Delta_0 \le \delta_0$.
\end{remark}

%%%%%%%%%%%%%%%%%%%%%%%%%%%%%%%%%%%%%%%%%%%%%%%%%%%%%%%%%%%%%%
%%%%%%%%%%%%%%%%%%%%%%%%%%%%%%%%%%%%%%%%%%%%%%%%%%%%%%%%%%%%%%%
\section{Adaptive  inexact fast augmented Lagrangian method}
In this section we analyze the overall iteration complexity of the IFAL Algorithm for an optimal choice of the smoothing parameter $\rho$ and then we
introduce an adaptive variant of this with the same computational complexity (up to a logarithmic factor) that is fully implementable in practice.

First, assume that we adopt the initialization strategy suggested in Remark \ref{remark1}  such that ${\Delta}_0 \le \delta_0$.
Therefore, using this strategy and the previous assumption that $\delta_k = \frac{\epsilon}{2(k+3)}$, the outer complexity can be estimated as:
\begin{equation}\label{outer_compl_optimal}
N^{\text{out}} \le \frac{1}{\epsilon^{1/2}} \max
\left\{(\epsilon/3)^{1/2}, (32/\rho)^{1/2},
\left(\frac{8\norm{x^*}}{\rho} + 2\sqrt{\frac{4 \epsilon }{3\rho}}
\right)^{1/2} \right\} = : \frac{\tilde{\gamma}}{\epsilon^{1/2}}.
\end{equation}
Note that the variation of the smoothing  parameter $\rho$ induces a trade-off between the number of the outer iterations and the
complexity of the inner subproblem, i.e. for a sufficiently large $\rho$ we have a single outer iteration, but a complex inner
subproblem. The next result provides an optimal choice for $\rho$ (up to a constant factor), such that the best total complexity is
obtained. For simplicity of the exposition, we assume $\norm{x^*}\ge 2$.

% Theorem 4.
\begin{theorem}\label{th_total_compl_optimal}
Let $\rho, \epsilon> 0, \mu_0 = \frac{4}{\rho}$ and $\tau_k   = \frac{2}{k+ 3}$.
Assume that, at each outer iteration $k$ of the IFAL($\rho,\epsilon$) Algorithm, Nesterov's optimal method \cite{Nes:04} is called for computing an
approximate solution $\tilde{u}_{\rho}(\hat{x}^k)$ of the inner subproblem \eqref{inner_subproblem} such that:
\begin{equation*}
\mathcal{L}_{0\rho}(\tilde{u}_{\rho}(\hat{x}^k),\hat{x}^k) - d_{\rho}(\hat{x}^k) \le \delta_k \quad \left(= \frac{\epsilon}{2(k+3)}\right).
\end{equation*}
Then,  by choosing the smoothing parameter:
\begin{equation*}
\rho = \frac{16 \norm{x^*}^2}{\epsilon},
\end{equation*}
for attaining an $\epsilon-$optimal solution of \eqref{problem}, the  IFAL Algorithm performs at most:
\begin{equation*}
\left \lfloor  \sqrt{\frac{6 L_f D_U^2}{\epsilon}} + \frac{4 \sqrt{6} D_U \norm{G} \norm{x^*}}{\epsilon} \right \rfloor
\end{equation*}
projections on the simple set $U$.
\end{theorem}

\begin{proof}
First, we observe that for  $\rho \ge \frac{16
\norm{x^*}^2}{\epsilon}$,  the IFAL Algorithm  performs a single outer iteration.
Indeed, from \eqref{outer_compl_optimal} one can obtain $\tilde{\gamma} \le \epsilon^{1/2}$ provided that:
\begin{equation*}
\rho \ge \frac{32}{\epsilon} \quad \text{and} \quad \frac{8
\norm{x^*}}{\rho} + 2 \sqrt{\frac{4 \epsilon}{3\rho}} \le \epsilon.
\end{equation*}
It can be seen that any $\rho$ such that $\rho \ge \frac{16
\norm{x^*}}{\epsilon}$ satisfies both conditions. In this case, when
a single outer iteration is sufficient,  the total computational
complexity is given by the inner complexity estimate \eqref{inner_Nesterov}:
\begin{equation*}
 \sqrt{\frac{6 L_f D_U^2}{\epsilon}} + \frac{4 \sqrt{6} D_U \norm{G} \norm{x^*}}{\epsilon}.
\end{equation*}
On the other hand, if $\rho \le \frac{16 \norm{x^*}^2}{\epsilon}$,
then from \eqref{outer_compl_optimal} we can further bound
$N^{\text{out}}$ as follows:
\begin{equation}\label{outer_bound_2}
  N^{\text{out}} \le \frac{2^{5/2} \norm{x^*}}{(\rho \epsilon)^{1/2}}.
\end{equation}
Using the same inner complexity estimate \eqref{inner_Nesterov} of  the
Nesterov's optimal method, the total computational complexity is
given by:
\begin{align*}
 \sum\limits_{k=0}^{N^{\text{out}} + 1} N^{\text{in}}_k
&\le 2 D_U \left[\frac{2(L_f + \rho \norm{G}^2)}{\epsilon} \right]^{1/2} (N^{\text{out}} + 4)^{3/2} \\
& \le 2(N^{\text{out}} + 4)^{3/2}\sqrt{\frac{2L_fD_U^2}{\epsilon}} +
2D_U\norm{G}(N^{\text{out}} + 4)^{3/2}\sqrt{\frac{2\rho}{\epsilon}}.
\end{align*}
Using \eqref{outer_bound_2} for the second term in the right hand
side of the above estimate, and optimizing over $\rho$ we obtain
that, for the optimal parameter $\rho^* = \frac{2^{11/3}
\norm{x^*}^2}{\epsilon}$, the necessary number of outer iterations
is at most 2. In conclusion, the choice $\frac{16
\norm{x^*}^2}{\epsilon}$ is  optimal up to a constant factor. \qed
\end{proof}

If one knows \textit{a priori} an upper bound on $\norm{x^*}$, then the previous result indicates that a proper choice of $\rho$  determines that a single outer iteration of the IFAL Algorithm to be sufficient to attain $\epsilon$-primal optimality and the overall iteration complexity is of order $\mathcal{O}(\frac{1}{\epsilon})$, which is better than $\mathcal{O}(\frac{1}{\epsilon^{5/4}})$ for any $\rho>0$.
However, in practice  $\norm{x^*}$ is unknown and thus the optimal smoothing parameter $\rho$ cannot be computed.
In order to cope with this problem, we provide further an implementable adaptive variant of the IFAL Algorithm, which has the
same total complexity given in Theorem \ref{th_total_compl_optimal} (up to a logarithmic factor).
The Adaptive Inexact Fast augmented Lagrangian (A-IFAL) algorithm relies  on a search procedure which is used typically  for penalty and augmented Lagrangian methods in the case when a bound on the optimal Lagrange multipliers is unknown (see, e.g., \cite{LanMon:14}).

\begin{center}
\framebox{
\parbox{11.5cm}{
\noindent \textbf{A-IFAL($\rho_0, \epsilon$) Algorithm}
%For $k \geq 1$ update:
\begin{itemize}
\item[]\textbf{Initialization:} Choose $\rho_0,\epsilon>0, \mu_0 = \frac{4}{\rho_0}$ and $(u^0,x^0)$ such that $\Delta_0 \le \delta_0$
\item[]\textbf{Iterations:} For $k=1,2,\dots$, perform:
\begin{enumerate}
\item Starting from $(u^{k-1},x^{k-1})$, run a single iteration of the IFAL$(\rho_k, \epsilon)$ Algorithm and obtain the output: $(u^k,x^k)$
\item If $ \norm{Gu^k + g} \le \epsilon$ then \textbf{STOP}; otherwise, update: $k = k+1, \rho_{k+1} = 2 \rho_k$ and return to step 2.
\end{enumerate}
\item[]\textbf{End}
\end{itemize}
}}
\end{center}
We notice that the A-IFAL  Algorithm can be regarded as a variant of the IFAL Algorithm with variable increasing 
smoothness parameter $\rho$ and constant inner accuracy $\delta = \frac{\epsilon}{6}$.
The following result provides the overall  complexity of the A-IFAL Algorithm 
necessary for attaining an $\epsilon-$optimal solution.

% Theorem 5.
\begin{theorem}
Let $\{ (u^k,x^k)\}$ be the sequences generated by the A-IFAL Algorithm.
Then, for attaining an $\epsilon-$optimal solution of \eqref{problem}, the following total number of
projections on $U$ need to be performed:
$$\log_2 \left(\frac{16 \norm{x^*}^2 }{\epsilon} \right)\sqrt{\frac{24L_f D_U^2}{\epsilon}} + \frac{80\sqrt{3} D_U \norm{G}\norm{x^*} }{\epsilon}.$$
\end{theorem}

\begin{proof}
We observe that the maximum number of outer iterations performed by the A-IFAL Algorithm is given by $N_{\max}^\text{out} = \log_2 \left(
\frac{16\norm{x^*}^2}{\epsilon\rho_0}\right)$. Thus the overall iteration complexity is given by:
\begin{align*}
 \sum\limits_{k=0}^{N_{\max}^\text{out}} N^{\text{in}}_k
&\le N_{\max}^\text{out} \sqrt{\frac{24L_f D_U^2}{\epsilon}} + \sum\limits_{k = 0}^{N_{\max}^\text{out}} \frac{2 \sqrt{6} D_U \norm{G} \rho_k^{1/2}}{\epsilon^{1/2}}\\
& = N_{\max}^\text{out} \sqrt{\frac{24L_f D_U^2}{\epsilon}} + \frac{2 \sqrt{6} D_U \norm{G} \rho_0^{1/2}}{\epsilon^{1/2}} \frac{\sqrt{2}^{N_{\max}^\text{out}+1} - 1}{\sqrt{2}-1}\\
& \le  \log_2 \left(\frac{16 \norm{x^*}^2 }{\epsilon}
\right)\sqrt{\frac{24L_f D_U^2}{\epsilon}} + \frac{80\sqrt{3} D_U
\norm{G}\norm{x^*} }{\epsilon}.
\end{align*}
\qed
\end{proof}

It is important to note  that  the adaptive A-IFAL algorithm has the same computational complexity as the
IFAL algorithm, up to a logarithmic factor.
Moreover, both algorithms are implementable, i.e., they can be stopped based on verifiable  stopping criteria and their parameters can be easily computed.

%%%%%%%%%%%%%%%%%%%%%%%%%%%%%%%%%%%%%%%%%%%%%%%%%%%%%%%%%%%%
%%%%%%%%%%%%%%%%%%%%%%%%%%%%%%%%%%%%%%%%%%%%%%%%%%%%%%%%%%%%%%%
\section{Comparison with other augmented Lagrangian complexity results}
In  this section we compare the computational complexity and other features of the IFAL/A-IFAL Algorithm with previous works and
complexity results on AL methods.

Given $x \in U, r>0$, we use  the notations $\mathcal{B}_r(x) = \{y \in \rset^n| \norm{y}\le r\} $  and
$\mathcal{N}_U(x) = \{y \in \rset^n | \langle y, z-x\rangle \le 0 \;\; \forall z \in U\}$.  Then,  in \cite{LanMon:14}, the authors analyze the classical
AL method for the same class of problems \eqref{problem}.
They developed an implementable variant of the classical
AL method to obtain an $\epsilon-$suboptimal primal-dual
pair $(u_{\epsilon},x_{\epsilon})$ satisfying the following
criteria:
\begin{equation}\label{lan_criteria}
\nabla f(u_{\epsilon}) + G^Tx_{\epsilon} \in
-\mathcal{N}_{U}(u_{\epsilon}) + \mathcal{B}_{\epsilon}(0) \quad
\text{and} \quad   \norm{G u_{\epsilon} + g} \le \epsilon.
\end{equation}
The authors provide their own iteration complexity analysis  for the augmented Lagrangian method using the  inexact dual  gradient
algorithm and without any artificial perturbation on the problem they obtained that it is necessary to perform
$\mathcal{O}\left(\frac{1}{\epsilon^{7/4}} \right)$ projections on the simple set $U$ in order to obtain a primal-dual pair satisfying
\eqref{lan_criteria}.
An important remark is that, for some $\rho>0$, the method in \cite{LanMon:14} requires \textit{a priori} a
pre-specified number of the outer iterations in order to compute the inner accuracy and to terminate the algorithm.
Moreover, to satisfy our $\epsilon-$optimality definition, an average primal iterate
$\hat{u}^k = \frac{1}{k}\sum_{i=0}^k u^i$ has to be computed (see \cite{NedNec:14}). Further,  for any fixed $\rho$, $\mathcal{O}(\frac{\norm{x^*}}{\rho \epsilon})$ outer
iterations has to be performed  and the inner accuracy has to be
chosen of the form $\delta_k  = \mathcal{O}(\frac{\rho^2
\epsilon^3}{\norm{x^*}})$. In this case, the method from
\cite{LanMon:14} requires:
\begin{equation}\label{total_lan}
 \left\lfloor \frac{2^7 D_U L_f^{1/2} \norm{x^*}^{3/2}}{\rho^{2}\epsilon^{5/2}} + \frac{2^7 \norm{G}D_U \norm{x^*}^{3/2}}{\rho^{3/2}\epsilon^{5/2}}\right\rfloor
\end{equation}
total projections on the set $U$, provided that  $\rho \le \frac{2^4 \norm{x^*}^{1/2}}{5 \epsilon}$.
However, we observe that for $\rho = \mathcal{O}(\norm{x^*}/\epsilon)$, from \eqref{total_lan} the overall iteration complexity is of order $\mathcal{O}(1/\epsilon)$ (as in the present paper), while for an arbitrary constant parameter $\rho$, the complexity estimates are much worse than those given in this paper.

It is important to note that although the inexact excessive gap methods introduced in \cite{QuoNec:15,QuoCev:14} are similar to the IFAL Algorithm and the authors also provide outer complexity estimates of order $\mathcal{O}\left( \frac{1}{\epsilon^{1/2}}\right)$ for a fixed
$\rho$, the update rules for the inner accuracy in \cite{QuoNec:15,QuoCev:14}  induce difficulties in the derivation of
the total complexity.
Moreover, assuming one implements the update rule of the inner accuracy $\delta_k$ given e.g. by \cite[Theorem
5.1]{QuoCev:14}, at each outer iteration it is required a primal iterate $\tilde{u}_{\rho}(x)$ satisfying:
$$ \mathcal{L}_{0\rho}(\tilde{u}_{\rho},x) - d_{\rho}(x) \le \frac{\rho \delta_k^2}{2}.$$
For a small constant $\rho$ and high accuracy, the  theoretical complexity of the algorithm can be very pessimistic.
Moreover, from our previous analysis we can conclude that for an  adequate choice of the parameter $\rho$ the number of outer iterations is $1$ and
therefore outer complexity estimates are irrelevant for the total complexity of the method.

Other recent complexity results concerning the classical
AL method are given, e.g., in \cite{AybIye:12,AybIye:13,NedNec:14}.
For example,  in \cite{AybIye:13} an adaptive classical AL method for cone constrained convex optimization models is analyzed.
However, the method in \cite{AybIye:13}  is not entirely implementable  (the stopping criteria cannot be verified) and the inner accuracy is constrained to be of order:
\begin{equation*}
\delta_k = \mathcal{O}\left(\frac{1}{k^2 \beta^k} \right),
\end{equation*}
where $\beta>1$. The authors in \cite{AybIye:13} show that the outer
complexity is of order $\mathcal{O}(\log(1/\epsilon))$ and thus, the
total complexity is similar with the estimates given  in our paper
(up to a logarithmic factor). However, our IFAL Algorithm  represents an accelerated augmented Lagrangian  method based on the excessive gap theory and moreover, it can be easily implemented in practice, i.e., it  can be stopped based on verifiable  stopping criteria and the parameters can be easily computed.

% Concluding remarks.
\section{Concluding remarks}
We have analyzed the iteration complexity of several inexact accelerated first-order augmented Lagrangian 
methods for solving linearly constrained convex optimization problems.
We have computed the optimal choice of the penalty parameter $\rho$ and
by means of smoothing techniques and excessive gap-like condition, we provided  estimates on the overall 
computational complexity of these algorithms. We compared our theoretical results 
with other existing results in the literature.
The implementation of these algorithms, numerical simulations, and comparison can be found in the 
forthcoming full paper.

%%%%%%%%%%%%%%%%%%%%%%%%%%%%%%%%%%%%%%%%%%%%%%%%%%%%%%%
%%%%%%%%%%%%%%%%%%%%%%%%%%%%%%%%%%%%%%%%%%%%%%%%%%%%%%%%

\end{document}